\documentclass[]{amsart}

\newtheorem{theorem}{Theorem}[section]

\newtheorem{remark}[theorem]{Remark}
\newtheorem{corollary}[theorem]{Corollary}

\newtheorem{proposition}[theorem]{Proposition}

\usepackage{graphicx,graphics}
\usepackage[all]{xy}
\newcounter{rmnum}

\numberwithin{equation}{section}

\begin{document}

  \baselineskip=17pt

\title{Characterizing maximal compact subgroups}

\author{Sergey A. Antonyan\\[3pt]}

\address{Departamento de  Matem\'aticas,
Facultad de Ciencias, Universidad Nacional Aut\'onoma de M\'exico,
 04510 M\'exico Distrito Federal, Mexico.}
\email{antonyan@unam.mx}

\begin{abstract} We prove that for a compact subgroup $H$ of an almost connected   locally compact Hausdorff group $G$, the following properties are mutually  equivalent: (1)~ $H$ is a maximal compact subgroup of $G$,  (2)~$G/H$ is  contractible, (3) $G/H$ is homeomorphic to a Euclidean space, (4)~$G/H$ is  an AE for paracompact spaces, (5)~$G/H$ is  a $G$-AE for paracompact proper $G$-spaces having a paracompact orbit space.

\end{abstract}


\dedicatory{Dedicated to the memory of Professor E.\,G. Skljarenko}
\maketitle \markboth{SERGEY ANTONYAN}{CHARACTERIZING MAXIMAL COMPACT SUBGROUPS}

\section{Introduction}

Throughout all topological groups are assumed to satisfy the  Huasdorff axiom of separation.

In 1979 B.\,Hoffmann~ \cite{hoff:79}  proved that the only  contractible compact  group is the trivial one. In this note we generalize this result as follows:
\begin{theorem}\label{T1} If $G$ is a compact  group and the closed subgroup $H\subset G$ is such that the coset space $G/H$ is contractible, then $H=G$.
\end{theorem}

 Then we apply  this theorem  to give the following characterization of maximal compact subgroups:

\begin{theorem}\label{T2} Let $G$ be a  locally compact almost connected group. Then a compact  subgroup $H\subset G$ is maximal compact if and only if  the coset space $G/H$ is contractible.
\end{theorem}

Recall that a locally compact group $G$ is called {\it almost connected}, if the space of connected components of
 $G$ is compact. A compact subgroup $K$ of $G$ is called {\it maximal compact} if every compact subgroup of $G$ is
  conjugate to a subgroup of $K$. By a well-known Malcev-Iwasawa theorem (see \cite[Ch.\,H, Theorem~32.5] {stroppel} for a convenient statement of the result; the proof can be found in the original literature referenced therein)), every locally compact almost connected group  $G$ has a maximal compact subgroup $K$. In this case  the coset space $G/K$ is homeomorphic to a Euclidean space.
    It follows from Theorem~\ref{T2} that the converse is also true:

\begin{corollary}\label{C1} Let $G$ be a  locally compact almost connected group. Then a compact  subgroup $H\subset G$ is maximal compact if and only if  the coset space $G/H$ is homeomorphic to a Euclidean space.
\end{corollary}

Based on Theorem~\ref{T2}, we give one more characterization of maximal compact subgroups in terms of (equivariant) extension properties of coset spaces. Namely, we have the following:

\begin{theorem}\label{T3}
Let $H$  be a compact  subgroup of a locally compact almost connected  group  $G$. Then the following properties are  equivalent:
\begin{enumerate}
\item $H$ is a maximal compact subgroup,
\item $G/H$ is a  $G$-{\rm AE}$(\mathcal P)$,
\item $G/H$ is an {\rm AE}$(\mathcal P)$.
\end{enumerate}
 \end{theorem}

Here $G/H$ denotes the quotient $G$-space of left cosets $\{xH\mid x\in G\}$ endowed with the $G$-action defined by left translations.
The notions of a $G$-{\rm AE}$(\mathcal P)$ and an {\rm AE}$(\mathcal P)$ involved in this theorem, as well as its proof, are given in section~4.
\medskip

\section{Proof of Theorem~\ref{T1}}

The proof is based on several well-known important results. The first of them is just a special  case of   Theorem~\ref{T1} when $G$ is a compact Lie group:

\begin{proposition}\label{P:0}
Let $H$  be a closed   subgroup of a compact Lie  group  $G$ such that the coset space $G/H$ is contractible. Then $H=G$.
 \end{proposition}
\begin{proof} The coset space  $G/H$ is a compact manifold. Assume that $\dim\, G/M=n\ge 1$. Then the $n$-th homology group $H_n(G/H, \Bbb Z_2)=\Bbb Z_2$ (see, e.g., the corollary of  \cite[Ch.~3, Theorem~3.26]{hatcher}) which, however, is impossible due to the contractibility of $G/H$.
Hence, it must be $n=0$, and then $G/H$ is a singleton, as required.
\end{proof}

\begin{remark} A close  argument was used by B. Hoffmann~ \cite{hoff:79} to prove the same  proposition for $H=\{e\}$, the trivial subgroup.

\end{remark}

The second fact we need  is the following  result of  J.~Szente~\cite[Theorem~4]{szente}:

\begin{proposition}\label{T:szente}
Let  $G$ be an almost connected group that acts effectively and transitively on a locally compact, locally contractible space. Then $G$ is a Lie group.
 \end{proposition}

Let $X$ be a space and  $a\in V\subset U\subset  X$. One says that $V$  is relatively contractible to  the point $a$ in $U$ provided that there exists a homotopy $F:V\times I\to U$ such that $F(v, 1)=v$, $F(v, 0)=a$ and $F(a, t)=a$ for all $v\in V$ and $t\in I=[0, 1]$. Such a contraction  is called relative with respect to   the point $a$.
A space $X$ is called relatively locally contractible if for any  point  $a\in X$ and any its neighborhood $U$,  there is a neighborhood $V$ of $a$ which  is relatively contractible to  $a$ in $U$.

The next result is perhaps the key tool in our argument:

\begin{proposition}\label{P:00}
Let $H$  be a compact   subgroup of a locally compact group  $G$ such that the coset space $G/H$ is contractible. Then $G/H$ is relatively locally contractible. \end{proposition}
\begin{proof}  Due to homogeneity it suffices to prove that $G/H$ is relatively locally  contractible at the point $eH$, the coset  of the unit element $e\in G$.

Let $F:G/H\times I\to G/H$ be a contraction such that $F(x, 1)=x$ and $F(x, 0)=eH$ for all $x\in G/H$.
 Due to a result of E.~Sklyarenko~\cite[Theorem 15]{skl:64},  the natural projection $\pi:G\to G/H$ is a Hurewicz fibration, i.e., it has the homotopy lifting property for arbitrary spaces. In particular, in de following commutative square      diagram there exists a filler, i.e., a diagonal continuous map $\psi$ which gives rise of  two triangle commutative diagrams:

$$\xymatrix{
 G/H\times\{0\}\ar[r]^{\qquad c}\ar@{_(->}[d]_i & G\ar[d]^\pi\\
G/H\times I\ar[r]_{\quad F}\ar[ru]_{\psi} & G/H
}$$
where  $c$ is the constant map to the point $e\in G$ and $i: G/H\times\{0\} \hookrightarrow G/H\times I$ is the standard  inclusion map.

Consider  the path $\varphi:I\to G$  defined by $\varphi(t)=\psi(eH, t)$, $t\in I$.

 Then  $\varphi(0)=e$,  $\varphi(1)\in H$  and $\varphi(t)H=F(eH, t)$ \ or, equivalently, $\varphi(t)^{-1}F(eH, t)=eH$ for every $t\in I$.

Next we define a new contraction $\Phi:G/H\times I\to G/H$ according the formula:
$$\Phi(x, t)=\varphi(1)\varphi(t)^{-1}F(x, t).$$
Then, for everey $(x, t)\in G/H\times I$,  we have:
$$\Phi(x, 1)=\varphi(1)\varphi(1)^{-1}F(x, 1)=F(x, 1)=x,$$
$$\Phi(x, 0)=\varphi(1)\varphi(0)^{-1}F(x, 0)=\varphi(1)\varphi(0)^{-1}eH=eH,$$
$$\Phi(eH, t)=\varphi(1)\varphi(t)^{-1}F(eH, t)=\varphi(1)eH=eH,\quad\forall t\in I.$$

Thus, $\Phi$ is a relative  contraction of $G/H$ to its point $eH$.

Next, assume that $U$ is any neighborhood of $eH$ in $G/H$. Since $\Phi(eH, t)=eH\in U$ for all $t\in I$, using compactness of  the unit interval $I=[0, 1]$,  one can find a neighborhood $V$ of $eH$ such that $\Phi(x, t)\in U$ for every $x\in V$. Thus, $\Phi: V\times I\to U$ is a relative contraction to the point $eH\in G/H$, and hence,   $G/H$ is relatively locally contractible at the point $eH\in G/H$.
\end{proof}

\medskip

\noindent
{\it Proof of Theorem~\ref{T1}.}
Denote by $N$ the kernel of the $G$-action on  $G/H$, i.e.,
$$N=\{g\in G  \mid   gx=x  \ \text{for all} \ x\in  G/H \}.$$
 Evidently,  $N\subset H$, and the compact group $G/N$ acts effectively and transitively on  $G/H$.

 Further, by Proposition~\ref{P:00}, the coset space $G/H$ is locally contractible.
   Consequently, one can apply  Proposition~\ref{T:szente} according to which  $G/N$ is a Lie group.

   Thus,    $H/N$  is a compact subgroup of the compact Lie group   $G/N$, and due to the natural homeomorphism
   \begin{equation}\label{homeo}
G/H\cong\frac{G/N}{H/N}
\end{equation}
 we infer  that the quotient space $\frac{G/N}{H/N}$ is contractible. Then Proposition~\ref{P:0} implies  that
 $G/N=H/N$. Since $ N$ is a subgroup of $H$ the latter equality quickly  yields  that $G=H$, as required.
\qed

\medskip

\section{Proof of Theorem~\ref{T2}}

Since $G$ is locally compact and almost connected, it has a maximal compact subgroup, say, $K$. Then $H$ is conjugated to a subgroup of $K$;  assume, without loss of generality, that $H\subset K$.

Consider $G$ as a  $K$-space  endowed with the $K$-action given by $k*g=gk^{-1}$ for all $k\in K$ and $g\in G$. Then $K$ is a $K$-invariant subset of $G$. Moreover,  it is immediate from the structure theorem for locally compact almost connected groups that  $K$ is a $K$-equivariant retract of $G$.
Indeed, according the structure theorem, there exists a  subset $E\subset G$ homeomorphic to some $\Bbb R^n$ such that the map  $(x, k)\mapsto xk: E\times K\to G$ is a homeomorphism. Every eleement $g\in G$, therefore, has a unique decomposition $g=x_gk_g$ with $x_g\in E$ and $k_g\in K$ (see \cite{hofterp:94}; for $G$ a Lie group see \cite[Ch.~XV, Theorem~3.1]{hoch:65}). Now the map $r:G\to K$ given by $r(g)=k_g$ is just the desired
  $K$-equivariant retraction.

  Further, since $H\subset K$ we see that $r$ is $H$-equivariant, and hence, it  induces a retraction $R:G/H\to K/H$. Next, since  by the hypothesis $G/H$ is contractible and a retraction preserves this property, we infer that $K/H$ is contractible. Now  Theorem~\ref{T1} implies  that $H=K$. But $K$ is a maximal compact subgroup of $G$, and hence,  $H$ is so,  as required.
  \qed

\medskip

\section{Proof of Theorem~\ref{T3}}

To start with, we  recall the  definition of a  proper action in the sense of R. Palais~\cite{pal:61}.

  A $G$-space   $X$ is called proper~ \cite[Definition 1.2.2]{pal:61}  if each point of $X$ has a, so called, {\it small} neighborhood, i.e., a neighborhood $V$ such that for every point of $X$ there is a neighborhood $U$ with the property that the set
$\langle U,V\rangle=\{g\in G \ | \  gU\cap V\not= \emptyset\}$    has compact closure in $G$.

Below we shall denote by  $G$-$\mathcal P$ the class of all paracompact proper $G$-spaces $X$ that  have paracompact orbit space $X/G$. It is a long time standing  open problem whether the orbit space of any paracompact proper $G$-space is paracompact (see  \cite{anne:03}).

 \medskip

A $G$-space  $Y$ is called an equivariant absolute    extensor  for the class $G$-$\mathcal P$   (notation:   $Y\in G$-AE($\mathcal P$))   if  for any $X\in G$-$\mathcal P$, every   $G$-map $f:A\to Y$ defined on a  closed invariant subset $A\subset X$,  extends to a $G$-map  $\psi\colon X\to Y$.  Taking here $G$ the trivial group we arrive to the well-known definition of a (nonequivariant) AE$(\mathcal P)$.

\medskip

\noindent
{\it Proof of Theorem~\ref{T3}}. $(1)\Longrightarrow (2)$ is known from \cite[Proposition~6(4)]{ant:99}.

\smallskip

$(2)\Longrightarrow (3)$. Suppose that $X$ is a paracompact space,  $A$ a closed subset of $X$, and $f:A\to G/H$  a continuous map. Consider the $G$-space $G\times X$ endowed with the action of $G$ defined by the rule:
$h(g, x)=(hg, x)$ for all $(g, x)\in G\times X$ and $h\in G$. Then the map $F:G\times A\to G/H$, given by $F(g, a)= gf(a)$, is a  $G$-map.

Since $G$ is a proper $G$-space,  so is the product $G\times X$.
Since the orbit space of $G\times X$ is  homeomorphic to $X$, we conclude that $G\times X\in G$-$\mathcal P$.
  Hence, by the hypothesis, $F$ extends to a  $G$-map $\widetilde F: G\times X\to G/H$.

Next we define a map $\widetilde f:X\to G/H$ by putting $\widetilde f(x)=\widetilde F(e, x)$. Clearly, $\widetilde f$ \ is a continuous extension of $f$, as required.

\smallskip

$(3)\Longrightarrow (1)$. Due to  Theorem~\ref{T2}, it suffices to show that  $G/H$ is contractible. Denote by $A$  the closed subset $G/H\times\{0\}\cup \{eH\}\times I\cup G/H\times\{1\}$ of the product $G/H\times I$. Consider the continuous map $f:A\to G/H$ defined by the rule:
$$f (u,0)=u  \ \ \text{and} \ \ f(u, 1)=eH \ \ \text{if}\  u\in G/H,\quad \text{and}\quad f(eH, t)=eH \ \ \text{for all} \  t\in I.$$
It is well-know that a locally compact group is paracompact (see, e. g., \cite[Ch.\,3, Theorem~3.1.1]{arhtk}). Since the projection $G\to G/H$ is a closed map and, by a theorem of E.~Michael, paracompactness is an invariant of closed maps, the coset  space $G/H$ is also paracompact. Then the product  $G/H\times I$ is paracompact, and since by the hypothesis $G/H\in {\rm AE}(\mathcal P)$, the map $f$  extends  to a continuous map  $F:G/H\times I\to G/H$, which is the desired contraction of $G/H$ to its point $eH$.
\qed

\bibliographystyle{amsplain}

\begin{thebibliography}{1110}


\bibitem{ant:99}
 S.~A.~Antonyan, {\em  Extensorial properties of orbit spaces of proper group actions},  Topol. Appl. {\bf  98} (1999), 35-46.

\bibitem{anne:03} S.~A.~Antonyan and S. de Neymet, {\em Invariant pseudometrics on Palais proper $G$-spaces}, Acta Math. Hung. {\bf 98 (1-2)} (2003), 41-51.

    \bibitem{arhtk}
A.~Arhangel'skii and M.~Tkachenko, {\em Topological Groups and Related Structures}, Atlantis Press/World Scientific, Amsterdam-Paris, 2008.



\bibitem{hoff:79}
 B.~Hoffmann,  {\em A compact contractible topological group is trivial},  Archiv Math.  {\bf 32, no. 1}  (1979),  585-587.

\bibitem{hofterp:94}
 K.~Hofmann and C. Terp,  {\em Compact subgroups of Lie groups and locally compact  groups},  Proc. AMS  {\bf 120, no. 2}
 (1994),  623-634.

\bibitem{hatcher}
 A.~Hatcher, {\em  Algebraic Topology}, Cambridge Univ. Press, 2001.

\bibitem{hoch:65}
 G.~Hochshild, {\em  The structure of Lie groups},  Holden-Day Inc., San Francisco,  1965


\bibitem{pal:61} R.~Palais,
{\em  On the existence of slices for actions of non-compact Lie
groups}, Ann. of Math. {\bf  73} (1961),  295-323.

\bibitem{skl:64} E.~G.~Skljarenko, {\em On the topological structure of locally bicompact groups and their quotient
spaces},  Amer. Math. Soc. Transl., Ser. 2,  {\bf  39}  (1964),
57-82.


\bibitem{stroppel} M. Stroppel,
{\em  Locally compact groups}, European Math. Soc., 2006.

\bibitem{szente} J. Szente,
{\em  On the topological characterization of transitive  Lie group actions},
Acta Sci. Math. Szeged. {\bf  36 (3-4)}, (1974),  323-344.


\end{thebibliography}

\end{document}